\documentclass[11pt]{article}

\usepackage{amssymb,amsmath,amsthm,latexsym,verbatim,url}
\usepackage{fullpage}
\makeatletter
\def\imod#1{\allowbreak\mkern10mu({\operator@font mod}\,\,#1)}
\makeatother

\newtheorem{thm}{Theorem}[section]
\newtheorem{cor}{Corollary}[section]
\newtheorem{lem}{Lemma}[section]
\theoremstyle{definition}
\newtheorem{defn}{Definition}[section]
\newtheorem{prop}{Proposition}[section]

\newtheorem*{thm*}{Theorem}
\newtheorem*{cor*}{Corollary}
\newtheorem*{lem*}{Lemma}
\theoremstyle{definition}
\newtheorem*{defn*}{Definition}

\DeclareMathOperator{\argmin}{argmin}

\DeclareMathOperator{\relint}{relint}
\DeclareMathOperator{\aff}{aff}

\DeclareMathOperator{\CC}{CC}

\newcommand{\set}[1]{\{#1\}}			
\newcommand{\floor}[1]{\lfloor #1 \rfloor}

\newcommand{\round}[1]{\lfloor #1 \rceil}
\newcommand{\pr}[2]{\langle #1, #2 \rangle}
\newcommand{\N}{\mathbb{N}}                     
\newcommand{\Q}{\mathbb{Q}}                     
\newcommand{\R}{\mathbb{R}}                     
\newcommand{\Z}{\mathbb{Z}}                     

\def\eps{\epsilon}

\def\int{\mathrm{int}}

\def\relbd{\mathrm{relbd}}
\def\bd{\mathrm{bd}}
\def\int{\mathrm{int}}
\def\relint{\mathrm{relint}}

\def\conv{\mathrm{conv}}
\def\ext{\mathrm{ext}}

\begin{document}

\title{On the Chv\'atal-Gomory Closure of a Compact Convex Set}
\author{Daniel Dadush\thanks{H. Milton Stewart School of Industrial and Systems Engineering, Georgia Institute of Technology, 765 Ferst Drive NW, Atlanta, GA 30332-0205, USA {\tt dndadush@gatech.edu}} \and Santanu S. Dey\thanks{H. Milton Stewart School of Industrial and Systems Engineering, Georgia Institute of Technology, 765 Ferst Drive NW, Atlanta, GA 30332-0205, USA {\tt santanu.dey@isye.gatech.edu}} \and Juan Pablo Vielma\thanks{Department of Industrial Engineering, University of Pittsburgh, 1048 Benedum Hall, Pittsburgh, PA 15261, USA {\tt jvielma@pitt.edu}} }
\maketitle
\thispagestyle{empty}
$\quad$
\begin{abstract}
In this paper, we show that the Chv\'atal-Gomory closure of a compact convex set is a rational polytope. This resolves an open question discussed in Schrijver~\cite{Schrijver80} and generalizes the same result for the case of rational polytopes~\cite{Schrijver80}, rational ellipsoids~\cite{Dey09} and strictly convex sets~\cite{Dadush:de:vi:10}. In particular, it shows that the CG closure of an irrational polytope is a rational polytope, which was the open question in~\cite{Schrijver80}. 
\end{abstract}
\pagebreak
\setcounter{page}{1}
\section{Introduction}
Gomory~\cite{Gomory58} introduced the Gomory fractional cuts, also known as Chv\'atal-Gomory (CG) cuts, to design the first finite cutting plane algorithm for integer linear programs. Since then, many important classes of facet-defining inequalities for combinatorial optimization problems have been identified as CG cuts. For example, the matching polytope can be obtained using Chv\'atal-Gomory cuts~\cite{edmonds:1965}. CG cuts have also been effective from a computational perspective; see for example~\cite{Bonami08b}, \cite{Fischetti07}. Although traditionally CG cuts have been defined for rational polyhedron for solving integer linear programs, they can be defined for general convex sets so as to be useful in solving convex integer programs, i.e. discrete optimization problems where the continuous relaxation is a convex optimization problem. CG cuts for non-polyhedral sets were considered implicitly in \cite{Chvatal73,Schrijver80} and more explicitly in \cite{CezikIyengar05,Dey09}. Let $K\subseteq \mathbb{R}^n$ be a closed convex set and let $h_K$ represent its support function, i.e. $h_K(a) = \textup{sup}\{ \langle a,x \rangle\,:\, x \in K\}$. Then given $a \in \mathbb{Z}^n$ such that $h_K(a) < +\infty$, the CG cut corresponding to $a$ is derived as,
\begin{eqnarray}\label{CGdefine}
\langle a, x \rangle \leq \lfloor h_K(a) \rfloor.
\end{eqnarray}
The CG closure is defined as the convex set obtained by the intersection of all viable CG cuts.
A classical result of Chv\'atal~\cite{Chvatal73} and Schriver~\cite{Schrijver80} states that the CG closure of a rational polyhedron is a rational polyhedron. This is a crucial property, since it is a mathematical guarantee that there exists a `relatively important' finite subset of CG cuts that defines the CG closure. Recently, we were able to verify that the CG closure of a compact convex set obtained as the intersection of a strictly convex set and a rational polyhedron is a rational polyhedron~\cite{Dey09,Dadush:de:vi:10}. The proof involved using significantly different techniques to the ones used in~\cite{Schrijver80}.

While intersection of strictly convex sets and rational polyhedron is an important class of convex sets, they do not capture the whole gamut of interesting convex sets that appear in convex IPs. The barrier in extending our understanding of the CG closure from the setting of the intersection of a strictly convex set and a rational polyhedra to the setting of a general convex set, is in dealing with irrationality. When working with integer linear programs, it is reasonable to assume that the set is defined by rational data and all the extreme points and rays of the feasible set are rational. However, when dealing with general convex IPs, this assumption breaks down in a natural way. For example, the Lorentz cone~\cite{nemibook} has irrational extreme rays and second order representable sets naturally (not always) inherit irrational generators. One way to design tools to deal with irrationality is to perhaps work with irrational polytopes. Schrijver~\cite{Schrijver80} considered this question. In a discussion section at the end of the paper, he writes that\footnote{Theorem 1 in~\cite{Schrijver80} is the result that the CG clsoure is a polyhedron. $P'$ is the notation used for CG closure in~\cite{Schrijver80}}:
\begin{quote}
``We do not know whether the analogue of Theorem 1 is true in real spaces. We were able to show only that if $P$ is a bounded polyhedron in real space, and $P'$ has empty intersection with the boundary of $P$, then $P'$ is a (rational) polyhedron."
\end{quote}

In this paper, we are able to prove the CG closure of any compact convex set\footnote{If the convex hull of a set of integer points is not a polyhedron, then the CG closure cannot be expected to be a rational polyhedron. Since we do not understand well when the convex hull of integer points in general convex sets is a polyhedron, we consider the question of CG closure only for compact convex sets here.} is a rational polytope, thus also resolving the question raised in~\cite{Schrijver80} for any polytope. We note here that while the intersection of the CG closure of a convex set $K$ and the boundary of $K$ is not always empty, this intersection identified in~\cite{Schrijver80} plays a crucial role in our proof.

As discussed above, proving that the compact convex set involved understanding and developing tools to handle irrationality. Therefore, while the proof presented in this paper is similar in parts to the proof in~\cite{Dadush:de:vi:10}, major components of the proof are new. New connections with diophantine approximations were necessary for the proof here. Moreover, we have been able to unearth some interesting new properties of CG closures and convex sets in general, and also design new techniques that we believe are important on their own.

This paper is organized as follows. In Section~\ref{def} we give some notation, formally state our main result and give an overview of the proof which is presented in  Sections~\ref{sec:long}--\ref{sec:final}.

\section{Definitions, main result and proof idea}\label{def}
\begin{defn}[CG Closure]For a convex set $K\subseteq \mathbb{R}^n$ and  $S\subseteq \mathbb{Z}^n$ let  $\CC(K,S) :=\bigcap_{a\in S} \set{x\in \mathbb{R}^n\,:\,  \pr{x}{y} \leq \floor{h_K(y)}}$. The CG Closure of $K$ is defined to be the set $\CC(K):=\CC(K,\mathbb{Z}^n)$.
\end{defn}
In this paper, we are able to establish the following result.
\begin{thm}\label{paperresult} If $K \subseteq \R^n$ be a non-empty compact convex set, then $\CC(K)$ is finitely generated. That is, there exists $S\subseteq \mathbb{Z}^n$ such that $|S|< \infty$ and $\CC(K)=\CC(K,S)$. In particular $\CC(K)$ is a rational polyhedron.
\end{thm}

We will use the following  notation in our proof:
\begin{itemize}
	\item Let $B^n = \set{x\in\mathbb{R}^n\,: \|x\| \leq 1}$ and $S^{n-1} = \bd(B^n)$. ($\bd$ stands for boundary)\vspace{-0.1in}

	\item For a convex set $K$ and $v\in S^{n-1}$ we let $H_v(K) := \set{x\in \R^n\,:\, h_K(v) = \pr{v}{x}}$ be the hyperplane defined by $v$ and the support function of $K$. We also let $F_v(K) := K\cap H_v(K)$ be the face of $K$ exposed by $v$. If $F_v(K)\neq K$ we say that $F_v(K)$ is a proper exposed face and if the context is clear we regularly drop $K$ from the notation and simply write $H_v$ and $F_v$.\vspace{-0.1in}

\item For $A \subseteq \R^n$, let $\aff(A)$ denote the smallest affine subspace containing $A$. Furthermore denote
$\aff_I(A) = \aff(\aff(A) \cap \Z^n)$, i.e. the largest integer subspace in $\aff(A)$.
\end{itemize}
This notation is fairly standard with the exception of $\aff_I(A)$. Understanding the properties of $\aff_I(A)$ when $\aff(A)$ is not a rational affine space will be crucial for the proof of Theorem~\ref{paperresult}. In particular, we will  repeatedly use the fact that if $K$ is a compact convex set, then we can obtain a inner approximation of $K\cap\aff_I(K)$ using a finite number of CG cuts.

The outline of the main steps in our proof of Theorem \ref{paperresult} is as follows:

\begin{enumerate}
	\item (Section \ref{sec:long}) For $v \in \R^n$ and $S\subseteq \mathbb{Z}^n$, show  that  $\exists S'\subseteq \mathbb{Z}^n$ such that $\CC(F_v,S) = H_v \cap \CC(K,S')$ and $|S|<\infty \Rightarrow |S'|< \infty$ by proving the following:
	\begin{enumerate}
    \item (Section \ref{sec:lift})  CG cuts  for  $F_v$ can be rotated or ``lifted''  to  become CG cuts for $K$ such that points in $F_v \cap \aff_I(H_v)$  separated by the original  CG cut for  $F_v$ are  separated by the new ``lifted'' one.
    \item (Section \ref{sec:sepirrational}) A finite number of CG cuts for $K$  separate all points in $F_v \setminus \aff_I(H_v)$.
    \end{enumerate}
\item (Section \ref{sec:approx}) Assuming $\CC(F_v)$ is finitely generated for any proper exposed face $F_v$ create an approximation $\CC(K,S)$ of $\CC(K)$ such that (i) $|S|<\infty$, (ii) $\CC(K,S)\subseteq K \cap \textup{aff}_I(K)$  (iii) $\CC(K,S) \cap \relbd(K) = \CC(K) \cap \relbd(K)$. This is done in the following two steps:
	\begin{enumerate}
		\item(Section \ref{sec:approx1}) Using the assumption, $\CC(F_v,S) = H_v \cap \CC(K,S')$ and a compactness argument create a first approximation satisfying (i) and (ii).
		\item(Section \ref{sec:approx2}) Using the assumption and noting that a polytope $P\subseteq K$ intersects $\relbd(K)$ along a finite number of faces of $P$ refine the approximation to  satisfy (iii).
	\end{enumerate}
\item (Section \ref{sec:final}) Finally, we are able to establish the result of the Theorem by induction on the dimension of $K$. The key observation is that there are only finitely many  CG cuts that separate at least one vertex of the second approximation of the CG closure.
\end{enumerate}
\vspace{-0.1in}
\section{$\CC(F_v,S) = H_v \cap \CC(K,S')$}\label{sec:long}
In the case of a rational polyhedra $K$, a key property of the CG closure  is that, if $F$ is a face of $K$, then $\CC(F) = F \cap \CC(K)$. Using an induction argument this property can be used  to construct the second approximation in the outline of our proof for the case in which $K$ is a rational polyhedron. However, this property is not enough for general convex sets.

For instance, when $K$ is a strictly convex set all proper faces of $K$ are single points and property  $\CC(F) = F \cap \CC(K)$ (or even $\CC(F) = F \cap \CC(K,S')$ for $|S'|<\infty$)  only tells us that every non-integral point in $\bd(K)$ can be separated with  CG cuts, but it does not tell us anything about the neighborhood of integral points. For this reason we need the stronger property $\CC(F_v,S) = H_v \cap \CC(K,S')$. In particular, this property implies that if $K$ is a full dimensional compact strictly convex and  $\CC(F_v)$ is finitely generated for every $v$ then for each integer point $x\in \bd(K)$ there exists a finite number of CG cuts that separate a neighborhood of $\bd(K)$ around $x$ which is exactly what is needed in \cite{Dadush:de:vi:10}.

We finally note here that the proof of the fact that $\CC(F_v,S) = H_v \cap \CC(K,S')$ for the case of general compact convex set is significantly more involved than for the case where $K$ is a rational polyhedron or a strictly convex set.

\subsection{Lifting CG cuts}\label{sec:lift}

$\CC(F) = F \cap \CC(K)$ is usually proven using a `lifting approach', i.e., given a  CG  for  $F$  of the form $\pr{w}{x} \leq \floor{h_{F}(w)}$ where $w \in \Z^n$, it is shown that there exists $w' \in \Z^n$ such that
\begin{eqnarray}\label{toprovelifting}
\set{x: \pr{w'}{x} \leq \floor{h_K(w')}} \cap \aff(F) \subseteq \set{x: \pr{w}{x} \leq \floor{h_{F}(w)}} \cap \aff(F).
\end{eqnarray}

In order to prove (\ref{toprovelifting}) (in the case of a rational polyhedron) we typically appeal to the rational description of $K$ and Farka's Lemma. The appropriate version of \eqref{toprovelifting} for strictly convex sets is proven in \cite{Dadush:de:vi:10} by approximating the left hand side of CG cuts for $F$ using Dirichlet's diophantine approximation theorem. The appropriate version of (\ref{toprovelifting}) for the case of general compact convex sets simply replaces $\aff(F)$ with $\aff_I(H_v)$ and generalizes both (\ref{toprovelifting}) and the version in  \cite{Dadush:de:vi:10}. This general version is given in  Proposition \ref{lem:lift2} with a proof that is similar to that in   \cite{Dadush:de:vi:10} and for which Dirichlet's  theorem again plays an important role.

	Lemmas \ref{lem:nopt-conv}- \ref{lem:2step-conv} are technical results that are needed for proving Proposition \ref{lem:lift2}.
\begin{lem} Let $K$ be a compact convex set in $\R^n$. Let $v \in \R^n$, and let $(x_i)_{i=1}^\infty$, $x_i \in K$,
be a sequence such that $\lim_{i \rightarrow \infty} \pr{v}{x_i} = h_K(v)$. Then
\[
\lim_{i \rightarrow \infty} d(F_v(K), x_i) = 0.
\]
\label{lem:nopt-conv}
\end{lem}
\begin{proof}
Let us assume that $\lim_{i \rightarrow \infty} d(F_v(K), x_i) \neq 0$. Then there exists an $\eps > 0$ such that for
some subsequence $(x_{\alpha_i})_{i=1}^\infty$ of $(x_i)_{i=1}^\infty$ we have that $d(F_v(K), x_{\alpha_i}) \geq \eps$.
Since $(x_{\alpha_i})_{i=1}^\infty$ is an infinite sequence on a compact set $K$, there exists a convergent subsequence
$(x_{\beta_i})_{i=1}^\infty$ where $\lim_{i \rightarrow \infty} x_{\beta_i} = x$ and $x \in K$. Now we note that
$d(F_v(K), x) = \lim_{i \rightarrow \infty} d(F_v(K), x_{\beta_i}) \geq \eps$, where the first equality follows from the
continuity of $d(F_v(K), \cdot)$. Since $d(F_v(K),x) > 0$ we have that $x \notin F_v(K)$. On the other hand,
\[
h_K(v) = \lim_{i \rightarrow \infty} \pr{v}{x_i} = \lim_{i \rightarrow \infty} \pr{v}{x_{\beta_i}} = \pr{v}{x}
\]
and hence $x \in F_v(K)$, a contradiction.
\end{proof}

\begin{lem} Let $K$ be a compact convex set in $\R^n$. Let $v \in \R^n$, and let $(v_i)_{i=1}^\infty$, $v_i \in
\R^n$, be a sequence such that $\lim_{i \rightarrow \infty} v_i = v$. Then for any sequence $(x_i)_{i=1}^\infty$,
$x_i \in F_{v_i}(K)$, we have that
\[
\lim_{i \rightarrow \infty} d(F_v(K), x_i) = 0.
\]
\label{lem:dir-conv}
\end{lem}
\begin{proof}
We claim that $\lim_{i \rightarrow \infty} \pr{x_i}{v} = h_K(v)$. Since $K$ is compact, there exists $R \geq 0$ such
that $K \subseteq RB^n$. Hence we get that
\begin{align*}
h_K(v) &= \lim_{i \rightarrow \infty} h_K(v_i) = \lim_{i \rightarrow \infty} \pr{v_i}{x_i}  \\
         &= \lim_{i \rightarrow \infty} \pr{v}{x_i} + \pr{v_i-v}{x_i}
         \leq \lim_{i \rightarrow \infty} \pr{v}{x_i} + \|v_i-v\|R = \lim_{i \rightarrow \infty} \pr{v}{x_i},
\end{align*}
where the first equality follows by continuity of $h_K$ ($h_K$ is convex on $\R^n$ and finite valued). Since each $x_i
\in K$, we get the opposite inequality $\lim_{i \rightarrow \infty} \pr{v}{x_i} \leq h_K(v)$ and hence we get equality
throughout. Now by lemma \ref{lem:nopt-conv} we get that $\lim_{i \rightarrow \infty} d(F_v(K),x_i) = 0$ as needed.
\end{proof}
In the next lemma, vector $w$ will eventually represent the left-hand-side of the CG cut for  $F_v$ that we want to lift and vectors $(s_i)_{i = 1}^{\infty}$ will represent a sequence of left-hand-side  vectors that will be used to derive ``lifted" CG cuts for $K$. The conditions given in Lemma \ref{lem:2step-conv} on $s_i$  will be achieved as a consequence of Dirichlet's approximation theorem applied to $v$ and the result of the lemma will allow the original and lifted CG cuts  to separate the same points in $\aff_I(F_v)$. 
\begin{lem} Let $K \subseteq \R^n$ be a compact convex set. Take $v, w \in \R^n$, $v \neq 0$. Let $(s_i,t_i)_{i=1}^\infty$,
$s_i \in \R^n, t_i \in \R_+$ be a sequence such that
\begin{equation}
\label{eq:2step-cd}
\begin{split}
a.&~~ \lim_{i \rightarrow \infty} t_i = \infty, \\
b.&~~ \lim_{i \rightarrow \infty} s_i - t_iv = w. \\
\end{split}
\end{equation}
Then for every $\eps > 0$ there exists $N_\eps \geq 0$ such that for all $i \geq N_{\eps}$
\begin{equation}
	\label{four}
h_K(s_i) + \eps \geq t_ih_{K}(v) + h_{F_v(K)}(w) \geq h_K(s_i) - \eps.
\end{equation}
\label{lem:2step-conv}
\end{lem}
\begin{proof}
By \ref{eq:2step-cd} (a,b) we have that
\begin{equation}
\label{eq:2step-dir}
\lim_{i \rightarrow \infty} \frac{s_i}{t_i} = v
\end{equation}
and that we may pick  $N_1 \geq 0$ such that
\begin{equation}
\label{eq:2step-pd}
\|s_i - t_iv\| \leq \|w\|+1 \leq C \quad \text{ for } i \geq N_1.
\end{equation}
Let $(x_i)_{i=1}^\infty$ be any sequence such that $x_i \in F_{s_i}(K)=F_{s_i/t_i}(K)$. For each $i \geq 1$, let $\tilde{x}_i =
\argmin_{y \in F_v(K)} \|x_i-y\|$. By (\ref{eq:2step-dir}) and Lemma \ref{lem:dir-conv}, we may pick $N_2 \geq 0$
such that
\begin{equation}
\label{eq:2step-gc}
d(F_v(K), x_i) = \|x_i-\tilde{x}_i\| \leq \frac{\eps}{2C} \quad \text{ for } i \geq N_2.
\end{equation}
Since $h_{F_v(K)}$ is a continuous function, we may pick $N_3 \geq 0$ such that
\begin{equation}
\label{eq:2step-of}
|h_{F_v(K)}(s_i - t_iv) - h_{F_v(K)}(w)| \leq \frac{\eps}{2} \quad \text{ for }i \geq N_3.
\end{equation}

Let $N_\eps = \max \set{N_1,N_2,N_3}$. Now since $x_i \in F_{s_i}(K)$ and $\tilde{x}_i \in F_v(K)$ we have that
\begin{equation}
\label{eq:2step-optc}
\pr{x_i}{s_i} \geq \pr{\tilde{x}_i}{s_i} \quad \text{ and } \pr{\tilde{x}_i}{t_iv} \geq \pr{x_i}{t_iv}.
\end{equation}
From (\ref{eq:2step-pd}), (\ref{eq:2step-gc}), (\ref{eq:2step-optc}) we get that for $i \geq N_\eps$
\begin{equation}
\label{eq:2step-nopt}
\begin{split}
\pr{x_i}{s_i} - \pr{\tilde{x}_i}{s_i} &\leq \pr{x_i}{s_i} - \pr{\tilde{x}_i}{s_i} +
			\pr{\tilde{x}_i}{t_iv} - \pr{x_i}{t_iv} = \pr{x_i-\tilde{x}_i}{s_i-t_iv} \\
	      &\leq \|x_i-\tilde{x}_i\|\|s_i-t_iv\| \leq \left(\frac{\eps}{2C}\right)C = \frac{\eps}{2}.
\end{split}
\end{equation}
From (\ref{eq:2step-nopt}) we see that for $i \geq N_\eps$
\begin{equation}
\label{eq:2step-nopt2}
h_K(s_i) \geq h_{F_v(K)}(s_i) \geq \pr{s_i}{\tilde{x}_i} \geq \pr{s_i}{x_i} - \frac{\eps}{2} = h_K(s_i) - \frac{\eps}{2}.
\end{equation}
Since $\pr{v}{\cdot}$ is constant on $F_v(K)$, we have that
\begin{equation}
\label{eq:2step-od}
h_{F_v(K)}(s_i) = h_{F_v(K)}(s_i-t_iv+t_iv) = h_{F_v(K)}(s_i-t_iv) + t_ih_{F_v(K)}(v) = h_{F_v(K)}(s_i-t_iv) + t_ih_K(v)
\end{equation}
Combining (\ref{eq:2step-of}), (\ref{eq:2step-nopt2}) and (\ref{eq:2step-od}) we get that for $i \geq N_\eps$,
\[
h_K(s_i) + \eps \geq t_ih_K(v) + h_{F_v(K)}(w) \geq h_K(s_i) - \eps
\]
as needed.
\end{proof}
\begin{thm}[Dirichlet's Approximation Theorem] Let $(\alpha_1,\ldots,\alpha_l) \in \R^l$. Then for every positive integer $N$, there exists $1 \leq n \leq N$ such that
$ \max_{1 \leq i \leq l} |n\alpha_i - \round{n\alpha_i}| \leq 1/N^{1/l}$
\label{lem:dir-approx}
\end{thm}

\begin{prop} Let $K \subseteq \R^n$ be  a compact and convex set, $v \in \R^n$ and $w \in \Z^n$. Then $\exists w' \in
\Z^n$ such that $\set{x: \pr{w'}{x} \leq \floor{h_K(w')}} \cap \aff_I(H_v(K))  \subseteq \set{x: \pr{w}{x} \leq \floor{h_{F_v(K)}(w)}}\cap \aff_I(H_v(K))$.
\label{lem:lift2}
\end{prop}

\begin{proof}
First, by possibly multiplying $v$ by a positive scalar we may assume that $h_K(v) \in \Z$. Let $S = \aff_I(H_v(K))$
. We may assume that $S \neq \emptyset$, since otherwise the statement is trivially true.

From Theorem~\ref{lem:dir-approx} for any $v \in \R^n$ there exists  $(s_i,t_i)_{i=1}^\infty$, $s_i \in \Z^n$, $t_i \in \N$ such that (a.) $t_i \rightarrow \infty$ and (b.) $\|s_i - t_iv\| \rightarrow 0$.  Now define the sequence $(w_i,t_i)_{i=1}^\infty$, where $w_i = w + s_i$, $i \geq 1$. Note that the sequence $(w_i,t_i)$ satisfies \eqref{eq:2step-cd} and hence by Lemma \ref{lem:2step-conv} for any $\eps > 0$, there exists $N_\eps$ such that \eqref{four} holds.  Let $\eps = \frac{1}{2}\bigl(1-(h_{F_v(K)}(w)-\lfloor h_{F_v(K)}(w)\rfloor)\bigr)$, and let $N_1 = N_\eps$. Note that $\floor{h_{F_v(K)}(w) + \eps} = \floor{h_{F_v(K)}(w)}$. Hence, since $h_K(v) \in \Z$ by assumption, for all $i \geq N_1$ we have that
\begin{equation*}
\floor{h_K(w_i)} \leq \floor{t_ih_K(v) + h_{F_v(K)}(w) + \eps} = t_ih_K(v) + \floor{h_{F_v(K)}(w) + \eps}
= t_ih_K(v) + \floor{h_{F_v(K)}(w)}.
\end{equation*}
Now pick $z_1,\dots,z_k \in S$ such that $\aff(z_1,\dots,z_k) = S$ and let  $R = \max \set{\|z_j\|: 1 \leq j \leq k}$. Choose $N_2$ such that $\|w_i - t_i v - w\| \leq \frac{1}{2R}$ for $i \geq N_2$. Now note that for $i \geq N_2$,
\begin{equation*}
|\pr{z_j}{w_i} - \pr{z_j}{t_iv + w}| = |\pr{z_j}{w_i - t_iv - w}| \leq \|z_j\|\|w_i - t_iv-w\| \leq
R \frac{1}{2R} = \frac{1}{2} \quad \forall j \in\{1,\dots,k\}.
\end{equation*}
Next note that since $z_j, w_i \in \Z^n$, $\pr{z_j}{w_i} \in \Z$. Furthermore, $t_i \in \N$, $\pr{v}{z_i} = h_K(v) \in \Z$ and $w \in \Z^n$ implies that $\pr{z_j}{t_i v + w} \in \Z$. Given this, we must have  $\pr{z_j}{w_i} = \pr{z_j}{t_i v + w} \quad \forall j \in [k],\,  i\geq 1$ and hence we get  $\pr{x}{w_i} = \pr{x}{t_i v + w} \quad \forall x \in S,\,  i\geq 1$.

Let $w' = w_i$ where $i = \max \set{N_1,N_2}$. Now examine the set $L = \set{x: \pr{x}{w'} \leq \floor{h_K(w')}} \cap S$. Here we get that $\pr{x}{w_i} \leq t_ih_K(v) +  \lfloor h_{F_v(K)}(w) \rfloor$ and $ \pr{x}{v} = h_K(v)$ for all $x \in L$
Hence, we see that $\pr{x}{w_i-t_iv} \leq \lfloor h_{F_v(K)}(w)\rfloor$ for all  $x \in L$.
Furthermore, since $\pr{x}{w_i-t_iv} = \pr{x}{w}$ for all $x \in L \subseteq \aff(S)$, we have that $\pr{x}{w} \leq \floor{h_{F_v(K)}(w)}$ for all $x \in L$, as needed.
\end{proof}

\subsection{Separating all point in $F_v \setminus \aff_I(H_v)$}\label{sec:sepirrational}

Replacing $\aff(F)$ by $\aff_I(H_v)$ in the generalization of \eqref{toprovelifting} strengthens property by replacing $F$ with $H_v$, but weakens it by replacing $\aff(\cdot)$ by $\aff_I(\cdot)$. Because of this we need to explicitly deal with the points  in $F_v\setminus \aff_I(H_v)$. In this section, we show that points in $F_v\setminus \aff_I(H_v)$ can be separated by using a finite number of CG cuts in Proposition \ref{prop:killirr}. In  prove this, we need the Kronecker simultaneous approximation theorem that is stated next. See Niven~\cite{niven1963} or Cassels~\cite{cassels1972} for a proof.
\begin{thm} Let $(x_1,\dots,x_n) \in \R^n$ be such that the numbers $x_1,\dots,x_n,1$ are linearly independent
over $\Q$. Then the set $\set{(nx_1 \pmod 1, \dots, nx_n \pmod 1): n \in \N}$ is dense in $[0,1)^n$.
\label{thm:dense}
\end{thm}
The following lemmas  conveniently normalize vector $v$ defining $F_v$ and $H_v$.

\begin{lem} Let $K \subseteq \R^n$ be a closed convex set, and let $T:\R^n \rightarrow \R^n$ be an invertible linear
transformation. Then $h_K(v) = h_{TK}(T^{-t}v)$  and $F_v(K) = T^{-1}(F_{T^{-t} v}(TK))$  for all $v \in \R^n$.
Furthermore, if $T$ is a unimodular transformation, then $ CC(K) = T^{-1}(CC(TK))$.
\label{lem:unimod}
\end{lem}
\begin{proof} Observe that
\[
h_{TK}(T^{-t}v) = \sup_{x \in TK} \pr{T^{-t}v}{x} = \sup_{x \in K} \pr{T^{-t}v}{Tx} = \sup_{x \in K} \pr{v}{x} = h_K(v).
\]
Now note that
\begin{align*}
T^{-1}(F_{T^{-t} v}(TK)) &= T^{-1} \left(~\set{x: ~x \in TK, ~h_{TK}(T^{-t}v) = \pr{T^{-t}v}{x}}~\right) \\
			 &= \set{x: ~Tx \in TK,~ h_{TK}(T^{-t}v) = \pr{T^{-t}v}{Tx}}
                            = \set{x: ~x \in K, ~h_K(v) = \pr{v}{x}} \\
                         &= F_v(K).
\end{align*}
Finally,
\begin{align*}
T^{-1}(CC(TK)) &= T^{-1}\left(\set{x: x \in TK, ~\pr{v}{x} \leq \floor{h_{TK}(v)} ~ \forall ~ v \in \Z^n}\right) \\
               &= \set{x: Tx \in TK, ~\pr{v}{Tx} \leq \floor{h_{TK}(v)} ~ \forall ~ v \in \Z^n} \\
               &= \set{x: Tx \in TK, ~\pr{T^{-t}v}{Tx} \leq \floor{h_{TK}(T^{-t}v)} ~ \forall v \in \Z^n} \\
               &= \set{x: x \in K, ~\pr{v}{x} \leq \floor{h_{K}(v)} ~ \forall v \in \Z^n} = CC(K).
\end{align*}
\end{proof}

\begin{lem} Take $v \in \R^n$. Then there exists an unimodular transformation $T: \R^n \rightarrow \R^n$ and $\lambda
\in \Q_{> 0}$ such that for $v' = \lambda T v$ we get that
\begin{equation}
\label{eq:irr-red}
v' =
  \left(~\underbrace{0,\dots,0}_{t ~ \mathrm{times}}, ~\underbrace{1}_{s ~ \mathrm{times}},~\alpha_1,\dots,\alpha_{r}\right),
\end{equation}
where $t,r \in \Z_+$, $s \in \set{0,1}$, and $\{1,\alpha_1,\dots,\alpha_r\}$ are linearly independent over $\Q$.
Furthermore, we have that $\mathcal{D}(v) = \inf \set{\dim(W): v \in W, W = \set{x \in \R^n: Ax = 0}, A \in \Q^{m \times n}} = s + r$.
\label{lem:irr-red}
\end{lem}
\begin{proof}
	
Choose a permutation matrix $P$ such that the rational entries of $Pa$ form a contiguous block starting from the first entry of $Pa$, i.e. let $k \in \set{0,\dots,n}$ such that $(Pa)_1,\dots,(Pa)_k \in \Q$ and $(Pa)_{k+1},\dots,(Pa)_n \in \R \setminus \Q$. Now we set our initial transformation $T \leftarrow P$, $\lambda \leftarrow 1$, and working vector $a' \leftarrow Pa$. In what follows, we will apply successive updates to $T$,$\lambda$ and $a'$ such that we maintain that $T$ is unimodular, $\lambda \in \Q_{> 0}$, and $a' = \lambda T a$.


First consider a vector $a' \in \R^n$ such that $a'_1,\dots,a'_k$ are rational and
$(1,a'_{k+1},\dots,a'_n)$ are linearly independent over $\Q$. If $k=0$, i.e. $(1, a'_1,\dots,a'_n)$ are linearly
independent over $\Q$, then we are done. We may therefore assume that $k \geq 1$. Similarly, if $(a'_1,\dots,a'_k) =
0^k$, then again we are done. Now let $a'_R = (a'_1,\dots,a'_k)$ and $a'_I = (a'_{k+1},\dots, a'_n)$. By our
assumptions, we note that $a'_R \neq 0$.  Via an appropriate scaling $\lambda' \in \Q_{> 0}$, we may achieve $\lambda'
a'_R \in \Z^k$ and $\mathrm{gcd}(\lambda' a'_1,\dots, \lambda' a_k) = 1$. Since $\lambda' \in \Q$, note that
$(1,a'_{k+1},\dots,a'_n)$ are linearly independent over $\Q$ iff $(1, \lambda' a'_{k+1},\dots, \lambda' a'_n)$ are. Set
$\lambda \leftarrow \lambda' \lambda$ and $a' \leftarrow \lambda' a'$. Next, applying the Euclidean algorithm on the
vector $a'_R$, we get a unimodular transformation $E$ such that
\[
E a'_R = (0^{k-1}, \mathrm{gcd}(a'_1,\dots,a'_k)) = (0^{k-1}, 1).
\]
Now define the unimodular transformation $T'$, where
\[
T'(x) = (E(x_1,\dots,x_k), x_{k+1},\dots,x_n).
\]
By construction, note that
$((Ta')_1,\dots,(Ta')_k) = E a'_R = (0^{k-1}, 1)$. Next note that $((Ta')_{k+1},\dots,(Ta')_n)$
are linearly independent over $\Q$. Letting $T \leftarrow T' T$ and $a' \leftarrow T' a'$, we have that
$a' = \lambda T a$ satisfies the required form.

Given the above case analysis, we are left with the case where $a'_R = (a'_1,\dots,a'_k) \in \Q^k$, $a'_I =
(a'_{k+1},\dots,a'_{n}) \in (\R \setminus \Q)^{n-k}$ and where $(1,a'_{k+1},\dots,a'_{n})$ have a
linear dependency over $\Q$. Now after an appropriate scaling of this dependency, we get numbers $c_0 \in \Q, c \in
\Z^{n-k} \setminus \set{0}$, $\mathrm{gcd}(c_1,\dots,c_{n-k}) = 1$, and where
\[
\pr{a_I}{c} = c_0
\]
Applying the Euclidean algorithm on $c$, we get a unimodular matrix $E$ such that
\[
E c = (\mathrm{gcd}(c_1,\dots,c_{n-k}), 0^{n-k-1}) = (1, 0^{n-k-1})
\]
Let $\hat{a} = E^{-t} a'_I$. Note that $E$ is unimodular iff $E^{-t}$ is unimodular. We get that
\[
\pr{a_I}{c} = c_0 \Rightarrow \pr{E^{-t} a_I}{E c} = c_0 \Rightarrow \hat{a}_1 = c_0
\]
Hence we see that $\hat{a}_1 = c_0 \in \Q$. Let $T'$ be the unimodular transformation defined by
\[
T'(x) = (x_1,\dots,x_k, E^{-t}(x_{k+1},\dots,x_n))
\]
Here $T'$ is the identity on the first $k$ coordinates, and acts like $E^{-t}$ on the last $n-k$ coordinates. Note that
$((T'a')_1,\dots,(T'a')_k) = (a'_1,\dots,a'_k) \in \Q^k$. Next $((T'a')_{k+1},\dots,(T'a')_n) = E^{-t} a'_I = \hat{a}$,
and $\hat{a}_1 \in \Q$. Hence $T'a'$ has at least one more rational coefficient than $a'$. By repeating the above operation suitable number of times, we obtain a vector $a' \in \R^n$ such that $a'_1,\dots,a'_k$ are rational and
$(1,a'_{k+1},\dots,a'_n)$ are linearly independent over $\Q$. By the previous analysis,
there exists unimodular transformation $T''$, $\lambda' \in \Q$ such that $\lambda'T''T'a'$ satisfies the required form. Letting $T \leftarrow T''T'T$, $\lambda \leftarrow \lambda'\lambda$, and $a' \leftarrow \lambda'T''T'a'$, we get the desired result.

For proving the second part of the result, we first claim that $\mathcal{D}(a') = \mathcal{D}(a)$. To see this, note that
\[
Aa' = 0 \Leftrightarrow A(\lambda Ta) = 0 \Leftrightarrow ATa =  0 \quad \text{ and } \quad
Aa = 0 \Leftrightarrow A\left(\frac{1}{\lambda} T^{-1} a'\right) = 0 \Leftrightarrow AT^{-1}a' =  0
\]
since $T$ is invertible and $\lambda \neq 0$. Since both $AT,AT^{-1}$ are rational, this gives that $\mathcal{D}(a') =
\mathcal{D}(a)$ as needed. Hence we need only show that $\mathcal{D}(a') = s+t$.

Take $y \in \Q^n$ such that $\pr{y}{a'} = 0$. Note that $a' = (0^t, 1^s, \alpha_1,\dots,\alpha_r)$ where
$(1,\alpha_1,\dots,\alpha_r)$ are linearly independent over $\Q$. If $s=0$, then $\sum_{i=1}^r y_{t+i}\alpha_i = 0$. Since $y \in \Q^n$, this gives a linear dependence of $(\alpha_1,\dots,\alpha_r)$ over $\Q$, and hence by assumption we must have that $y_{t+i} = 0$ for $1 \leq i \leq r$. Otherwise if $s=1$, we get $y_{t+1} + \sum_{i=1}^r y_{t+i+1}\alpha_i = 0$, which gives a linear dependence of $(1,\alpha_1,\dots,\alpha_r)$ over $\Q$. Therefore $y_{t+i} = 0$ for $1 \leq i \leq t+1$. Hence in both cases, we get that $y_{t+i} = 0$ for $1 \leq i \leq r+s$. Next note that for $y \in \Q^t \times 0^{n-t}$, we have that $\pr{y}{a'} = 0$ since $a'_1,\dots,a'_r = 0$ by assumption. By the previous observations, we obtain that
\[
L := \set{y \in \Q^n: \pr{y}{a'} = 0} = \Q^t \times 0^{n-t} = \Q^t \times 0^{s+r}.
\]
Now let $W \subseteq \R^n$ denote the linear subspace $W = \set{x \in \R^n: x_i = 0, 1 \leq i \leq t}$. Note that $a'
\in W$, and hence $\mathcal{D}(a') \leq \dim(W) = s+r$. Now take any $M = \set{x \in \R^n: Ax = 0}$, such that $a' \in
M$ and $A \in \Q^{m \times n}$. We claim that $W \subseteq M$. Let $a_1,\dots,a_m \in \Q^n$ denote the rows of $A$. Since
$a' \in M$, we have $\pr{a_i}{a'} = 0 ~ \forall ~ i \in \{1,\dots,m\}$. Hence we must have that $a_i \in L = \Q^t \times 0$. Since
$W = 0^t \times \R^{s+r}$, we have that for all $x \in W$, $\pr{a_i}{x} = 0$, and hence $W \subseteq L$. Hence
\[
\dim(L) \geq \dim(W) = s+r,
\]
from which conclude that $\mathcal{D}(a') = s+r$ as needed.
\end{proof}
We now show that points belonging to $F_v\setminus \aff_I(H_v)$ can be separated by using a finite number of CG cuts. The proof can be viewed as follows: We select $\mathcal{D}(v)+1$ vectors whose conic span is the linear subspace corresponding to the irrational components. Using each of these directions as guides, we scale the vector $v$ (corresponding to the face $F_v$) by integers and use the Kronecker theorem to compute a tiny ``correction vector" to be added to the scaled version of $v$. In this way we produce $\mathcal{D}(v)+1$ integer vectors that are very close in angle to $v$. These integer vectors have the property that the CG cuts corresponding to them separate points in $F_v\setminus \aff_I(H_v)$. In all this, Lemma \ref{lem:unimod} is crucial as it allows to simplify the choice of case analysis.

\begin{prop}\label{prop:killirr} Let $K \subseteq \R^n$ be a compact convex set and $v \in \R^n$.
Then there exists $C \subseteq \Z^n$, $|C| \leq \mathcal{D}(v)+1$, such that
\begin{alignat*}{3}
\CC(K,C) \cap H_v(K) &\subseteq \aff_I(H_v(K))\\
\CC(K,C)&\subseteq \set{x: \pr{v}{x} \leq h_K(v)}.
\end{alignat*}
\end{prop}
\begin{proof}
By scaling $v$ by a positive scalar if necessary, we may assume that $h_K(v) \in \set{0,1,-1}$. Let $T$ and $\lambda$
denote the transformation and scaling promised for $v$ in Lemma \ref{lem:irr-red}. Note that
\[
T^{-t} \set{x \in \R^n: \pr{v}{x} = h_K(v)} = \set{x \in \R^n: \pr{v}{T^tx} = h_K(v)} = \set{x \in \R^n: \pr{\lambda
Tv}{x} = h_{T^{-t}K}(\lambda Tv)}.
\]
Now let $v' = \lambda Tv$ and $b' = h_{T^{-t}K}(\lambda Tv)$. By Lemma \ref{lem:unimod}, it suffices  to prove
 the statement for $v'$ and $K' = T^{-t} K$. Now $v'$ has the
  form \eqref{eq:irr-red} where $t,r \in \Z_+$, $s \in \set{0,1}$, and $(1,\alpha_1,\dots,\alpha_r)$ are linearly independent over $\Q$. For
convenience, let $k = s+t$, where we note that $v'_{k+1},\dots,v'_{k+r} = (\alpha_1,\dots,\alpha_r)$.

\paragraph{Claim 1:} Let $S = \set{x \in \Z^n: \pr{v'}{x} = b'}$. Then $S$ satisfies one of the following
\begin{enumerate}
\item $S = \Z^t \times b' \times 0^r$: $s=1, b' \in \Z$.
\item $S = \Z^t \times 0^r$: $s=0, b' = 0$.
\item $S = \emptyset$: $s=0, b' \neq 0$ or $s=1, b' \notin \Z$.
\end{enumerate}
Note that $b' = h_{T^{-t}K}(\lambda Tv) = \lambda h_K(v) \in \set{0, \pm  \lambda} \subseteq \Q$. We first see that
\[
(s=1):~ b'= \pr{v'}{x} = x_k + \sum_{i=1}^r x_{k+i}\alpha_i, \quad (s=0):~ b'=\pr{v'}{x} = \sum_{i=1}^r x_{k+i}\alpha_i.
\]
Now if $x \in S$, then
\[
(s=1):~ (x_k-b') + \sum_{i=1}^r x_{k+i}\alpha_i = 0, \quad (s=0): (-b') + \sum_{i=1}^r x_{k+i}\alpha_i = 0.
\]
Since $b' \in \Q$, and $x \in \Z^n$, in both cases the above equations give us a linear dependence of
$(1,\alpha_1,\dots,\alpha_r)$ over $\Q$. Since by assumption $(1,\alpha_1,\dots,\alpha_r)$ are linearly independent over
$\Q$, we have that
\[
(s=0,1):~ x_{k+i} = 0, 1 \leq i \leq r \quad (s=1):~ x_k = b' \quad (s=0):~ b' = 0.
\]
If $s=1$, then we must have that $b' \in \Z$, since $x_k = b'$ and $x \in \Z^n$. From this we immediately recover case
$(1)$. If $s=0$, then the conditions $b' = 0$ and $x_{k+i} = 0$, $1 \leq i \leq r$, verify case $(2)$. If we
are neither in case $(1)$ or $(2)$, then by the above analysis $S$ must be empty, and so we are done.
\paragraph{Claim 2:} Let $I = \set{n v' \pmod 1: n \in N}$. Then Theorem \ref{thm:dense} implies that $I$ is dense in $0^k \times [0,1)^r$.\linebreak
We first note that $v'_1,\dots,v'_k \in \Z$ and hence $v'_1,\dots,v'_k \equiv 0 \pmod 1$. Next note that
$(1,\alpha_1,\dots,\alpha_r)$ are linearly independent over $\Q$, and hence by Theorem \ref{thm:dense}
we have that $\set{n (\alpha_1,\dots,\alpha_r): n \in N}$ is dense over $[0,1)^r$. Putting the last two
statements together immediately yields the claim.

\paragraph{Claim 3:} There exists $a_1,\dots,a_{r+1} \subseteq \Z^n$ and  $\lambda_1,\dots,\lambda_{r+1} \geq 0$ such that $\sum_{i=1}^{r+1} \lambda_i a_i = v'$ and  $\sum_{i=1}^{r+1} \lambda_i\floor{h_K'(a_i)} \leq b'$.

\quad\linebreak
\indent Since $K'$ is compact, there exists $R > 0$ such that $K' \subseteq RB^n$. Take the subspace $W = 0^k \times \R^r$.  Let
$w_1,\dots,w_{r+1} \in W \cap S^{n-1}$, be any vectors such that for some $0 < \eps < 1$ we have $\sup_{1 \leq i \leq r+1} \pr{w_i}{d} \geq \eps$  for all $d \in
S^{n-1} \cap W$ (e.g. $w_1,\dots,w_{r+1}$ are the vertices of a scaled isotropic $r$-dimensional simplex).
Let $a = \frac{1}{8} \min \set{\frac{1}{R}, \eps}$, and $b = \frac{1}{2}\eps a$. Now, for $1 \leq i \leq r+1$ define $E_i = \set{x: x \in a w_i + b (B^n \cap W) \pmod 1}$.
Since $W = 0^k \times \R^r$, note that $E_i \subseteq 0^k \times [0,1)^r$. By Claim 2 the set $I$ is dense in $0^k
\times [0,1)^r$. Furthermore each set $E_i$ has non-empty interior with respect to the subspace topology on $0^k \times
[0,1)^r$. Hence for all $i$, $1 \leq i \leq r + 1$, we can find $n_i \in \N$ such that $n_iv' \pmod 1 \in E_i$.

Now $n_iv' \pmod 1 \in E_i$, implies that for some $\delta'_i \in E_i$, $n_iv' - \delta'_i \in \Z^n$. Furthermore
$\delta'_i \in E_i$ implies that there exists $\delta_i \in a w_i + b(B^n \cap W)$ such that $\delta_i' - \delta_i \in
\Z^n$. Hence $(n_iv' - \delta'_i) + (\delta'_i - \delta_i) = n_iv' - \delta_i \in \Z^n$. Let $a_i = n_iv' - \delta_i$.
Note that $ \|a_i - n_iv'\| = \|-\delta_i\| \leq a + b \leq 2a \leq 1/(4R)$.  We claim that $\floor{h_{K'}(a_i)} \leq h_{K'}(n_iv')$. First note that $h_{K'}(n_iv') = n_ib'$. Since we assume
that $S \neq \emptyset$, we must have that $b' \in \Z$ and hence $n_ib' \in \Z$. Now note that
\begin{align*}
h_{K'}(a_i) &= h_{K'}((a_i-n_iv')+n_iv')
            \leq h_{K'}(n_iv') + h_{K'}(a_i-n_iv')
            = n_ib' + h_{K'}(-\delta_i) \\
            &\leq n_ib' + h_{RB^n}(-\delta_i)
            \leq n_ib' + R\|\delta_i\|
            \leq n_ib' + R\left(\frac{1}{4R}\right) = n_ib' + \frac{1}{4}.
\end{align*}
Therefore we have that $\floor{h_{K'}(a_i)} \leq \floor{n_ib' + \frac{1}{4}} = n_ib' = h_{K'}(n_iv')$,
since $n_ib' \in \Z$.

We claim that $\frac{a\eps}{4}B^n \cap W \subseteq \conv \set{\delta_1,\dots,\delta_{r+1}}$. First note that by
construction, $\conv\set{\delta_1,\dots,\delta_{r+1}} \subseteq W$. Hence if the conclusion is false, then by the separator
theorem there exists $d \in W \cap S^{n-1}$ such that $h_{\frac{a\eps}{4}B^n \cap W}(d) = \frac{a\eps}{4} > \sup_{1 \leq i \leq r+1} \pr{d}{\delta_i}$.
For each $i$, $1 \leq i \leq r+1$, we write $\delta_i = a w_i + bz_i$ where $\|z_i\| \leq 1$. Now note that
\begin{align*}
\sup_{1 \leq i \leq r+1} \pr{d}{\delta_i} &= \sup_{1 \leq i \leq r+1} \pr{d}{a w_i + bz_i}
                                           = \sup_{1 \leq i \leq r+1} a \pr{d}{w_i} + b\pr{d}{z_i} \\
                                          &\geq \sup_{1 \leq i \leq r+1} a\pr{d}{w_i} - b\|d\|\|z_i\|
                                          \geq a \eps - b = \frac{a\eps}{2} > \frac{a\eps}{4},
\end{align*}
a contradiction. Hence there exists $\lambda_1,\dots,\lambda_{r+1} \geq 0$ and $\sum_{i=1}^{r+1} \lambda_in_i = 1$ such that $\sum_{i=1}^{r+1} \lambda_i \delta_i = 0$.

Now we see that
\begin{equation}
\sum_{i=1}^{r+1} \lambda_ia_i = \sum_{i=1}^{r+1} \lambda_in_iv' + \sum_{i=1}^{r+1} \lambda_i(a_i-n_iv')
			      = \left(\sum_{i=1}^{r+1} \lambda_in_i\right) v' - \sum_{i=1}^{r+1} \lambda_i \delta_i
                              = \left(\sum_{i=1}^{r+1} \lambda_in_i\right) v'.
\end{equation}
Next note that
\begin{equation}
\sum_{i=1}^{r+1} \lambda_i\floor{h_{K'}(a_i)} \leq \sum_{i=1}^{r+1} \lambda_ih_{K'}(n_iv') =
h_{K'}\left(\left(\sum_{i=1}^{r+1} \lambda_in_i\right) v'\right).
\end{equation}
\paragraph{Case 2: $S = \emptyset$.} \hspace{1em}
\noindent
The proof here shall proceed very similarly to the one above, with the exception that we need to do some extra work to
guarantee a strict inequality.

If $s=0$, then since $S = \emptyset$ we must have that $b' \neq 0$. Let $v^z = \frac{1}{|b'|}v'$ and $b^z =
\mathrm{sign}(b')$, and $v^f = \frac{1}{2|b'|}v'$ and $b^f = \frac{1}{2}\mathrm{sign}(b')$. Note that $h_{K'}(v^z) = b^z
\in \set{\pm1}$ and $h_{K'}(v^f) = b^f \in \set{\pm 1/2}$. Furthermore, since $b' \in \Q$, we see that
\[
(1,v^z_{k+1},\dots,v^z_{k+r}) = (1,\frac{1}{2|b'|}\alpha_1,\dots,\frac{1}{2|b'|}\alpha_r)
\]
are still linearly independent over $\Q$, and that $v^z_1,\dots,v^z_k = v'_1,\dots,v'_k= 0 \in \Z$.

Next if $s = 1$, then $b' \in \Q \setminus \Z$. Let $c_1 \in \Z$ denote the least positive integer such that $c_1b' \in \Z$ and let $c_2 \in \Z$ denote the least positive integer such that $\frac{1}{3} \leq c_2b' \pmod 1 \leq
\frac{2}{3}$ (always exists since $b' \neq 0$). Let $v^z = c_1v'$ and $b^z = c_1b'$, and let $v^f= c_2v'$ and $b^f = c_2b'$.  Again we
have that $h_{K'}(v^z) = b^z \in \Z$, and $h_{K'}(v^f) = b^f$ (since $c_1,c_2 \geq 0$). Lastly, since $c_1,c_2 \in \Z$,
we note that $v^z_1,\dots,v^z_{k-1} = v^f_1,\dots, v^f_{k-1} = 0 \in \Z$, $v^z_k = c_1, v^f_k = c_2 \in \Z$, and
$(1,v^z_{k+1},\dots,v^z_{k+r}) = (1,c_1\alpha_1,\dots,c_1\alpha_r)$ are still linearly independent over $\Q$.

Now let $I' = \set{nv^z \pmod 1: n \in \N}$. Using the proof of Claim $2$, we see that $I'$ is dense in $0^k \times
[0,1)^r$.  Furthermore since $v^f \mod 1 \in 0^k \times [0,1)^r$, we have that $I' + v^f \pmod 1$ is also dense in $0^k
\times [0,1)^r$. Note that $I' + v^f \pmod 1 = \set{(nc_1 + c_2)v' \pmod 1: n \in \N}$.

Let $w_1,\dots,w_{l+1}$, $E_1,\dots,E_{l+1}$ be defined identically as in Case 1. Via the same density argument as in
case $1$, we may pick $n_i \in \N$, such that $(n_ic_1 + c_2)v' \in E_i$. Again we define $a_1,\dots,a_{r+1}$ in exactly the same way as in Case $1$. To conclude the proof of the claim, we need only show that $\floor{h_{K'}(a_i)} \leq \floor{n_ib' + \frac{1}{4}} = n_ib' = h_{K'}(n_iv')$ holds with a strict inequality
 in this case. The exact same argument gives us now that
\begin{equation}
h_{K'}(a_i) \leq (n_ic_1+c_2)b' + \frac{1}{4}.
\end{equation}
Now $n_ic_1b' = n_ib^z \in \Z$  and $\frac{1}{3} \leq c_2b' \pmod 1 \leq \frac{2}{3}$. Therefore
\begin{equation}
\floor{h_{K'}(a_i)} < (n_ic_1+c_2)b',
\end{equation}
as needed.
\paragraph{Claim 4:} Let $C = \{a_i\}_{i=1}^{r+1}$ for the $a_i$'s from Claim 3. Then
$ \CC(K,C) \cap \set{x: \pr{v'}{x} = b'} \subseteq \aff(S)$.
If $S = \emptyset$, note that by the Claim $3$, we have that
\[
\sup \set{\pr{v'}{x}: x \in \R^n, \pr{a_i}{x} \leq \floor{h_{K'}(a_i)}, 1 \leq i \leq r+1} < b',
\]
and hence $\CC(K,C) \cap \set{x: \pr{v'}{x} = b'} = \emptyset$ as needed.

If $S \neq \emptyset$, examine the set

\quad\linebreak
\indent Examine the set $P = \set{x: \pr{v'}{x} = b', \pr{a_i}{x} \leq \floor{h_{K'}(a_i)}, 1 \leq i \leq l+1}$.
From the proof of Claim $3$, we know that for each $i$, $1 \leq i \leq r+1$, we have  $\floor{h_{K'}(a_i)} \leq h_{K'}(n_iv') = n_ib'$
and hence $\pr{n_iv' - a_i}{x} = \pr{\delta_i}{x} \geq 0$,
is a valid inequality for $P$. Now, from the proof of Claim $3$, we have
\begin{equation}
\frac{a\eps}{4}B^n \cap W \subseteq \conv \set{\delta_1,\dots,\delta_{r+1}}.
\label{eq:lift1-bc}
\end{equation}
We claim that for all $H \subseteq \set{1,\dots,r+1}$, $|H| = r$, the set $\set{\delta_i: i \in H}$ is linearly
independent. Assume not, then WLOG we may assume that $\delta_1,\dots,\delta_r$ are not linearly independent.
Hence there exists $d \in S^{n-1} \cap W$, such that $\pr{d}{\delta_i} = 0$ for all $1 \leq i \leq n$. Now
by possibly switching $d$ to $-d$, we may assume that $\pr{d}{\delta_{r+1}} \leq 0$. Hence we get that $\sup_{1 \leq i \leq r+1} \pr{d}{\delta_i} \leq 0$
in contradiction to (\ref{eq:lift1-bc}).

Now let $\lambda_1,\dots,\lambda_{r+1} \geq 0$, $\sum_{i=1}^{r+1}
\lambda_in_i = 1$ be a combination such that $\sum_{i=1}^{r+1} \lambda_i \delta_i = 0$. Note that
$\lambda_1,\dots,\lambda_{r+1}$ forms a linear dependency on $\delta_1,\dots,\delta_{r+1}$, and hence by
the previous claim we must have that $\lambda_i > 0$ for all $1 \leq i \leq r+1$.

We claim for $P\subseteq W^\perp$. To see this, note that $ 0 = \pr{x}{0} = \pr{x}{\sum_{i=1}^{r+1} \lambda_i \delta_i} = \sum_{i=1}^{r+1} \lambda_i \pr{x}{\delta_i}$ for every $x\in P$.
Now since $\mathrm{span}(\delta_1,\dots,\delta_{r+1}) = W$, we see that $\pr{x}{\delta_i} = 0$ for all $1 \leq i \leq r+1$ iff $x \in W^\perp$. Hence if $x \notin W^\perp$, then by the above equation and the fact that $\lambda_i > 0$ for all $i \in \set{1,\dots,r + 1}$, there exists $i,j \in \set{1,\dots,r + 1}$ such that $\pr{x}{\delta_i} > 0$ and $\pr{x}{\delta_j} < 0$. But then $x \notin P$, since $\pr{x}{\delta_j} < 0$, a contradiction.
Now $W = 0^k \times \R^r$, hence $W^\perp = \R^k \times 0^r$. To complete the proof we see  that $P  \subseteq \set{x: x \in \R^k \times 0^r, \pr{v'}{x} = b'} = \aff(S)$.
\end{proof}
\subsection{Combining results of Section \ref{sec:lift} and \ref{sec:sepirrational} to show $\CC(F_v,S) = H_v \cap \CC(K,S')$}\label{sec:facehomogenity}

\begin{prop}\label{PPPP}\label{thm:cg-lift} Let $K \subseteq \R^n$ be a compact convex set. Take $v \in \R^n$. Assume that $\CC(F_v(K))$ is finitely
generated. Then $\exists~ S \subseteq \Z^n$, $|S| < \infty$, such that $\CC(K,S)$ is a polytope and
\begin{alignat}{3}
\label{331} \CC(K,S) \cap H_v(K) &= \CC(F_v(K))\\	
\label{332} \CC(K,S) &\subseteq \set{x: \pr{v}{x} \leq h_K(v)}.	
\end{alignat}
\
\end{prop}
\begin{proof}
	The right to left containment in \eqref{331} is direct from  $\CC(F_v(K))\subseteq \CC(K,S)$ as every CG  cut for $K$ is a CG cut for $F_v(K)$. For the reverse containment and for \eqref{332} we proceed as follows.

 Using Proposition \ref{prop:killirr} there exists $S_1 \subseteq \Z^n$ such that $\CC(K,S_1) \cap H_v(K) \subseteq \aff_I(H_v(K))$
and $ \CC(K,S_1) \subseteq \set{x: \pr{v}{x} \leq h_K(v)}$.  Next let $G \subseteq \Z^n$ be such that $\CC(F_v(K), G) = \CC(F_v(K))$. For
each $w \in G$, by Proposition \ref{lem:lift2} there exists $w' \in \Z^n$ such that
\begin{equation*}
\CC(K, w') \cap \aff_I(H_v(K)) \subseteq \CC(F_v(K), w) \cap \aff_I(H_v(K)).
\end{equation*}
For each $w \in G$, add $w'$ above to $S_2$. Now note that
\begin{align*}
\CC(K,S_1 \cup S_2) \cap H_v(K)  &= \CC(K,S_1) \cap \CC(K,S_2) \cap H_v(K) \\
                                           &\subseteq \CC(K,S_2) \cap \aff_I(H_v(K)) = \CC(F_v(K), G) \cap \aff(A) \subset \CC(F_v(K)).
\end{align*}
Now let $S_3 = \set{\pm e_i: 1 \leq i \leq n}$. Note that since $K$ is compact $\CC(K,S_3)$ is a cuboid with bounded
side lengths, and hence is a polytope. Letting $S = S_1 \cup S_2 \cup S_3$, yields the desired result.
\end{proof}

We also obtain a generalization of the classical result known for rational polyhedra.

\begin{cor} If $F$ is an exposed face of $K$ then $\CC(F)=\CC(K)\cap F$.
\end{cor}
\section{Approximation of the CG closure}\label{sec:approx}
\subsection{Approximation 1 of the CG closure}\label{sec:approx1}

In this section, we construct our first approximation of the CG closure. Under the assumption that the CG closure of every proper exposed face of $K$ is defined by a finite number of CG cuts and by the use of Proposition~\ref{PPPP} and a compactness argument we construct a first approximation of the CG closure that uses a finite number of CG cuts. The main properties of this approximation are that it is a polytope and it is contained in $K \cap \textup{aff}_I(K)$. For this we will need the following  lemma that describes integer affine subspaces.

\begin{lem} Take $A \in \R^{m \times n}$ and $b \in \R^m$. Then there exists $\lambda \in \R^m$ such that for $a' =
\lambda A$, $b' = \lambda b$, we have that $\set{x \in \Z^n: Ax = b} = \set{x \in \Z^n: a'x = b'}$.
\label{lem:eqsys-red}
\end{lem}
\begin{proof}
If $\set{x \in \R^n: Ax = b} = \emptyset$, then by Farka's Lemma there exists $\lambda \in \R^m$ such that $\lambda A =0$ and $\lambda b = 1$. Hence $\set{x \in \R^n: Ax = b} = \set{x \in \R^n: 0x = 1} = \emptyset$ as needed. We may
therefore assume that $\set{x \in \R^n: Ax = b} \neq \emptyset$. Therefore we may also assume that the rows of the augmented matrix $[A\,|\,b]$ are linearly independent.

Let $T = \mathrm{span}(a_1,\dots,a_m)$, where $a_1,\dots,a_m$ are the rows of $A$. Define $r: T \rightarrow \R$ where
for $w \in T$ we let $r(w) = \lambda b$  for $\lambda \in \R^m$  where $\lambda A = w$.
Since the rows of $A$ are linearly independent we obtain that $r$ is well defined and is a linear operator.
Let $S = \set{x \in \Z^n: Ax = b}$. For $z \in \Z^n$, examine $T_z = \set{w \in T: \langle w, z \rangle = r(w)}$. By linearity of $r$, we
see that $T_z$ is a linear subspace of $T$. Note that for $z \in \Z^n$, $T_z = T$ iff $z \in S$. Therefore $\forall ~ z \in \Z^n \setminus S$, we must have that $T_z \neq T$, and hence $\dim(T_z) \leq \dim(T) - 1$. Let $m_T$ denote the Lebesgue measure on $T$. Since $\dim(T_z) < \dim(T)$, we see that $m_T(T_z) = 0$. Let $T' = \bigcup_{z \in \Z^n \setminus
S} T_z$. Since $\Z^n \setminus S$ is countable, by the countable subadditivity of $m_T$ we have that $m_T\left(T' \right) \leq \sum_{z \in \Z^n \setminus S} m_T(T_z) = 0$.
Since $m_T(T) = \infty$, we must have that $T \setminus T' \neq \emptyset$.  Hence we may pick $a' \in T \setminus T'$.
Letting $b' = r(a')$, we note that by construction there $\exists ~ \lambda \in \R^m$ such that $\lambda A = a'$ and
$\lambda b = b'$. Hence for all $z \in S$, $\lambda A z = \lambda b \Rightarrow a' x = b'$.  Now take $z \in \Z^n
\setminus S$. Now since $a' \in T \setminus T'$, we have that $a' \notin T_z$. Hence $a' z \neq b'$. Therefore we see
that $\set{x \in \Z^n: a' x = b'} = \set{x \in \Z^n: A x = b}$
as needed.
\end{proof}
\begin{prop} Let $\emptyset\neq K \subseteq \R^n$ be a  compact convex set. If  $\CC(F_v(K))$ is finitely generated for any proper exposed face $F_v(K)$ then $\exists~ S \subseteq \Z^n$, $|S| < \infty$, such that $\CC(K,S) \subseteq K \cap \aff_I(K)$ and $\CC(K,S)$ is a polytope.
\label{lem:move-inside}
\end{prop}
\begin{proof}
Let us express $\aff(K)$ as $\set{x \in \R^n: Ax = b}$. Note that $\aff(K) \neq \emptyset$ since $K \neq \emptyset$. By Lemma \ref{lem:eqsys-red} there exists $\lambda$, $c = \lambda A$ and $d = \lambda b$, and such that $\aff(K) \cap \Z^n = \set{x \in \Z^n: \pr{c}{x} = b}$.
Since $h_K(c) = b$ and $h_K(-c) = -b$, using Proposition \ref{prop:killirr} on $c$ and $-c$, we can find $S_A \subseteq \Z^n$
such that $\CC(K,S_A) \subseteq \aff(\set{x \in \Z^n: \pr{c}{x} = b}) = \aff_I(K)$.

Express $\aff(K)$ as $W + a$, where $W \subseteq \R^n$ is a linear subspace and $a \in \R^n$. Now take $v \in W \cap
S^{n-1}$. Note that $F_v(K)$ is a proper exposed face  and hence, by assumption,
$\CC(F_v(K))$ is finitely generated.  Hence by Proposition \ref{thm:cg-lift} there exists $S_v \subseteq \Z^n$ such that $\CC(K,S_v)$ is a polytope, $\CC(K,S_v) \cap H_v(K) = \CC(F_v(K))$ and $\CC(K,S_v) \subseteq \set{x: \pr{x}{v} \leq h_K(v)}$.  Let $K_v = \CC(K,S_v)$, then we have the following claim.

\quad\linebreak
\textbf{Claim:} $\,\exists$   open neighborhood $N_v$ of
$v$ in $W \cap S^{n-1}$ such that $ v' \in N_v \Rightarrow h_{K_v}(v') \leq h_K(v')$.

\quad\linebreak
Since $K_v$ is a polytope, there exists $C \subseteq \R^n$, $|C| < \infty$, such that $K_v = \conv (C)$.  Then note
that $h_{K_v}(w) = \sup_{c \in C} \pr{c}{w}$. Now let $H = \set{c: h_K(v) = \pr{v}{c}, c \in C}$.  By construction, we
have that $\conv (H) = \CC(F_v(K))$.

First assume that $\CC(F_v(K)) = \emptyset$. Then $H = \emptyset$, and hence $h_{K_v}(v) < h_K(v)$. Since $K_v,K$ are
compact convex sets, we have that $h_{K_v},h_K$ are both continuous functions on $\R^n$ and hence $h_K - h_{K_v}$ is
continuous. Therefore there exists $\eps > 0$ such that $h_{K_v}(v') < h_K(v')$ for $\|v-v'\| \leq \eps$ as needed.

Now assume that $\CC(F_v(K)) \neq \emptyset$. Let $R = \max_{c \in C} \|c\|$, and let
\[
\delta = h_K(v) - \sup \set{\pr{v}{c}: c \in C \setminus H}.
\]
Now let $\eps = \frac{\delta}{2R}$. Now take any $v'$ such that $\|v'-v\| < \eps$. Now for all $c \in H$, we have that
\[
\pr{c}{v'} = \pr{c}{v} + \pr{c}{v'-v} = h_K(v) + \pr{c}{v'-v} \geq h_K(v) - \|c\|\|v'-v\| > h_K(v) -
R\frac{\delta}{2R} = h_K(v) - \frac{\delta}{2},
\]
and that for all $c \in C \setminus H$, we have that
\[
\pr{c}{v'} = \pr{c}{v} + \pr{c}{v'-v} \leq h_K(v) - \delta + \pr{c}{v'-v} \leq h_K(v) -\delta + \|c\|\|v'-v\|
< h_K(v) - \delta + \frac{\delta}{2} = h_K(v) - \frac{\delta}{2}.
\]
Therefore we have that $\pr{c}{v'}>\pr{c'}{v'}$ for all $c\in H$, $c'\in C \setminus H$ and hence
\begin{equation}
h_{K_v}(v') = \sup_{c \in C} \pr{c}{v'} = \sup_{c \in H} \pr{c}{v'} = h_{\CC(F_v(K))}(v') \leq h_K(v'),
\end{equation}
since $\CC(F_v(K)) \subseteq F_v(K) \subseteq K$. The statement thus holds by letting $N_v=\{v' \in S^{n-1}\,:\,\|v'-v\| \leq \eps\}$.
Note that $\set{N_v: v \in W \cap S^{n-1}}$ forms an open cover of $W \cap S^{n-1}$, and since $W \cap S^{n-1}$ is
compact, there exists a finite subcover $N_{v_1},\dots,N_{v_k}$ such that $\bigcup_{i=1}^k N_{v_i} = W \cap S^{n-1}$.
Now let $S = S_A ~\cup ~ \cup_{i=1}^k S_{v_i}$. We claim that $\CC(K,S) \subseteq K$. Assume not, then there exists $x \in
\CC(K,S) \setminus K$. Since $\CC(K,S) \subseteq \CC(K,S_A) \subseteq W+a$ and $K \subseteq W+a$, by the
separator theorem there exists $w \in W \cap S^{n-1}$ such that $h_K(w) = \sup_{y \in K} \pr{y}{w} < \pr{x}{w} \leq h_{\CC(K,S)}(w)$.
Since $w \in W \cap S^{n-1}$, there exists $i$, $1 \leq i \leq k,$ such that $w \in N_{v_i}$. Note then we obtain that
\begin{equation*}
h_{\CC(K,S)}(w) \leq h_{\CC(K,S_{v_i})}(w) = h_{K_{v_i}}(w) \leq h_K(w),
\end{equation*}
a contradiction. Hence $\CC(K,S) \subseteq K$ as claimed. $\CC(K,S)$ is a polytope because it is the intersection of polyhedra or which at least one is a polytope.
\end{proof}

\subsection{Approximation 2 of the CG closure}\label{sec:approx2}
In this section, we augment  the first approximation of $\CC(K)$ by finitely more CG cuts to construct a better approximation of $\CC(K)$. Apart from satisfying the condition that this approximation is contained in $K \cap \textup{aff}_I(K)$, it also satisfies the condition that its intersection with the relative boundary of $K$ is equal to the intersection of $\CC(K)$ with the relative boundary of $K$.

To achieve this approximation, the key observation is that since the first approximation of the CG closure was a polytope, therefore its intersection with relative boundary of $K$ is the union of a finite numbers of faces of
the first approximation of the CG closure. This implies that there are a finite number of faces of $K$ such that if we apply Proposition \ref{thm:cg-lift} to them (i.e. separates points in $F_v \setminus \aff_I(H_v)$ and add lifted version of the CG cuts for $F_v$), we are able to achieve the second approximation of the CG closure.

\begin{lem} Let $K \subseteq \R^n$ be a convex set and $P \subseteq K$ be a polytope. Then there exists $F_{v_1},\dots,F_{v_k}
	\subseteq K$, proper exposed faces of $K$, such that $ P \cap \relbd(K) ~\subseteq~ \bigcup_{i=1}^k F_{v_i}$
\label{lem:finite-bd}
\end{lem}
\begin{proof}
Let $\mathcal{F} = \set{F: F \subseteq P, F \text{ a face of P }, \relint(F) \cap \relbd(K) \neq \emptyset}$. Since $P$
is polytope, note that the total number of faces of $P$ is finite, and hence $|\mathcal{F}| < \infty$. We claim that
\begin{equation}
P \cap \relbd(K) \subseteq \bigcup_{F \in \mathcal{F}} F.
\end{equation}
Take $x \in P \cap \relbd(K)$. Let $F_x$ denote the minimal face of $P$ containing $x$ (note that $P$ is a face
of itself). By minimality of $F_x$, we have that $x \in \relint(F_x)$. Since $x \in \relbd(K)$,
we have that $F_x \in \mathcal{F}$, as needed.

Take $F \in \mathcal{F}$. We claim that there exists $H_F \subseteq K$, $H_F$ a proper exposed face of $K$,
such that $F \subseteq H_F$. Take $x \in \relint(F) \cap \relbd(K)$. Let $\aff(K) = W + a$, where $W$ is a linear
subspace and $a \in \R^n$. Since $x \notin \relint(K)$, by the separator theorem, there exists $v \in W \cap S^{n-1}$
such that $h_K(v) = \pr{x}{v}$. Let $H_F = F_v(K)$. Note that since $v \in W \cap S^{n-1}$, $F_v(K)$ is a proper exposed
face of $K$. We claim that $F \subseteq H_F$. Since $F$ is a polytope, we have that $F = \conv(\ext(F))$. Write $\ext(F) =
\set{c_1,\dots,c_k}$. Now since $x \in \relint(F)$, there exists $\lambda_1,\dots,\lambda_k > 0$, $\sum_{i=1}^k
\lambda_i = 1$, such that $\sum_{i=1}^k \lambda_i c_i = x$. Now since $c_i \in K$, we have that $\pr{c_i}{v} \leq
h_K(v)$. Therefore, we note that
\begin{equation}
\pr{x}{v} = \pr{\sum_{i=1}^k \lambda_i c_i}{v} = \sum_{i=1}^k \lambda_i \pr{c_i}{v} \leq \sum_{i=1}^k \lambda_i h_K(v) =
 h_K(v)
\end{equation}
Since $\pr{x}{v} = h_K(v)$, we must have equality throughout. To maintain equality, since $\lambda_i > 0$, $1 \leq i
\leq k$, we must have that $\pr{c_i}{v} = h_K(v)$, $1 \leq i \leq k$. Therefore $c_i \in H_F$, $1 \leq i \leq k$,
and hence $F = \conv(c_1,\dots,c_k) \subseteq H_F$, as needed.

To conclude the proof, we note that the set $\set{H_F: F \in \mathcal{F}}$ satisfies the conditions of the lemma.
\end{proof}

\begin{prop} \label{lem:inter-body}
Let $K \subseteq \R^n$ be a compact convex set. If  $\CC(F_v)$ is finitely generated for any proper exposed face $F_v$ then  $\exists ~ S \subseteq \Z^n$, $|S| < \infty$, such that
	\begin{alignat}{4}
	\label{eqn56}\CC(K,S) &\subseteq K \cap \aff_I(K)\\
\label{eqn56b}\CC(K,S) \cap \relbd(K) &= \CC(K) \cap \relbd(K)
\end{alignat}
\end{prop}
\begin{proof}
By Proposition \ref{lem:move-inside}, there exists $S_I \subseteq \Z^n$, $|S_I| < \infty$, such that $\CC(K,S_I) \subseteq K
\cap \aff_I(K)$ and $\CC(K,S_I)$ is a polytope.  Since $\CC(K,S_I) \subseteq K$ is a polytope, let $F_{v_1},\dots,F_{v_k}$ be the proper exposed faces of $K$ given by Lemma~\ref{lem:finite-bd}.
By Proposition \ref{thm:cg-lift}, there
exists $S_i \subseteq \Z^n$, $|S_i| < \infty$, such that $\CC(K,S_i) \cap H_{v_i} = \CC(F_{v_i})$.
Let $S = S_I \cup \cup_{i=1}^k S_i$. We claim that $\CC(K,S) \cap \relbd(K) \subseteq \CC(K) \cap \relbd(K)$. For this note that $x
\in \CC(K,S) \cap \relbd(K)$ implies  $x \in \CC(K,S_I) \cap \relbd(K)$, and hence there exists $i$, $1 \leq i
\leq k$, such that $x \in F_{v_i}$. Then
\begin{align*}
	x \in \CC(K,S) \cap H_{v_i} \subseteq \CC(K,S_i) \cap H_{v_i}
	= \CC(F_{v_i}) \subseteq \CC(K) \cap \relbd(K).
\end{align*}
The reverse inclusion is direct.
\end{proof}

\section{Proof of Theorem}\label{sec:final}

Finally, we have all the ingredients to prove the main result of this paper. The proof is by induction on the dimension of $K$. Trivially, the result holds for zero dimensional convex body. Now by the induction hypothesis, we are able to construct the second approximation of $\CC(K)$ described in Section \ref{sec:approx2} (since it assumes that the CG closure of every exposed face is a polytope). Now the key observation is that any CG cut that is not dominated by those already considered in the second approximation of the CG closure must separate a vertex of this second approximation that additionally lies  in the relative interior of $K$. Then it is not difficult to show that there can exist only a finite number of such CG cuts, showing that the CG closure is a polytope. This proof idea is similar to a proof idea used in the case strictly convex sets. 

\begin{thm} Let $K \subseteq \R^n$ be a non-empty compact convex set. Then $\CC(K)$ is finitely generated.
\end{thm}
\begin{proof}
We proceed by induction on the affine dimension of $K$. For the base case, $\dim(\aff(K)) = 0$, i.e. $K = \set{x}$ is a
single point. Here it is easy to see that setting $S = \set{\pm e_i: i \in [n]}$, we get that $\CC(K,S) = \CC(K)$. The
base case thus holds.

Now for the inductive step let $0 \leq k < n$ let $K$ be a compact convex set where $\dim(\aff(K)) = k+1$ and assume the result holds for sets of lower dimension. By the induction hypothesis, we know that $\CC(F_v)$ is finitely generated for every proper
exposed face $F_v$ of $K$, since $\dim(F_v)\leq k$. By
Proposition \ref{lem:inter-body}, there exists a set $S \subseteq \Z^n$, $|S| < \infty$, such that \eqref{eqn56} and \eqref{eqn56b} hold.
If $\CC(K,S) = \emptyset$, then we are done. So assume that $\CC(K,S) \neq \emptyset$. Let $A = \aff_I(K)$. Since $\CC(K,S) \neq \emptyset$, we have that $A \neq \emptyset$ (by (\ref{eqn56})), and so we may pick $t \in A \cap \Z^n$. Note that $A - t = W$, where $W$ is
a linear subspace of $\R^n$ satisfying $W = \mathrm{span}(W \cap \Z^n)$. Let $L = W \cap \Z^n$. Since $t \in \Z^n$, we easily
see that $\CC(K-t,T) = \CC(K,T) - t$  for all $T  \subseteq \Z^n$.  Therefore $\CC(K)$ is finitely generated iff $\CC(K-t)$ is. Hence replacing $K$ by $K-t$, we may assume that
$\aff_I(K) = W$.

Let $\pi_W$ denote the orthogonal projection onto  $W$. Note that for all $x \in W$, and $z \in \Z^n$,
we have that $\pr{z}{x} = \pr{\pi_W(z)}{x}$. Now since $\CC(K,S) \subseteq K \cap W$, we see that for all $z \in \Z^n$
\begin{equation*}
\CC(K,S \cup \set{z}) = \CC(K,S) \cap \set{x: \pr{z}{x} \leq \floor{h_K(z)}} = \CC(K,S) \cap \set{x: \pr{\pi_W(z)}{x}
\leq \floor{h_K(z)}}.
\end{equation*}
Let $L^* = \pi_W(\Z^n)$. Since $W$ is a rational subspace, we have that $L^*$ is full dimensional lattice in $W$. Now
fix an element of $w \in L^*$ and  examine $V_w := \set{\floor{h_K(z)}: \pi_W(z) = w, z \in \Z^n}$.
Note that $V_w \subseteq \Z$. We claim that $\inf(V_w) \geq -\infty$. To see this, note that
\begin{align}
\inf \set{\floor{h_K(z)}: \pi_W(z) = w, z \in \Z^n} &\geq \inf \set{\floor{h_{K \cap W}(z)}: \pi_W(z)=w, z \in \Z^n} \\
                                 &= \inf \set{\floor{h_{K \cap W}(\pi_W(z))}: \pi_W(z)=w, z \in \Z^n} \\
                                 &= \floor{h_{K \cap W}(w)} > -\infty.
\end{align}
Now since $V_w$ is a lower bounded set of integers, there exists $z_w \in \pi^{-1}_W(w) \cap \Z^n$ such that $\inf(V_w) =
\floor{h_K(z_w)}$. From the above reasoning, we see that $ \CC(K,S \cup \pi^{-1}_W(z) \cap \Z^n) = \CC(K, S \cup \set{z_w})$.
Now examine the set $C = \set{w: w \in L^*, \CC(K,S \cup \set{z_w}) \subsetneq \CC(K,S)}$.
Here we get that
\begin{equation*}
\CC(K) = \CC(K, S \cup \Z^n) = \CC(K, S \cup \set{z_w: w \in L^*}) = \CC(K, S \cup \set{z_w: w \in C}).
\end{equation*}
From the above equation, if we show that $|C| < \infty$, then $\CC(K)$ is finitely generated. To do this, we will
show that there exists $R > 0$, such that $C \subseteq RB_n$, and hence $C \subseteq L^* \cap RB_n$. Since $L^*$ is a
lattice, $|L^* \cap RB_n| < \infty$ for any fixed $R$, and so we are done.

Now let $P = \CC(K,S)$. Since $P$ is a polytope, we have that $P = \conv(\ext(P))$. Let $I = \set{v: v \in \ext(P), v
\in \relint(K)}$, and let $B = \set{v: v \in \ext(P), v \in \relbd(K)}$. Hence $\ext(P) = I \cup B$. By assumption on
$\CC(K,S)$, we know that for all  $v \in B$, we have that $v \in \CC(K)$. Hence for all $z \in \Z^n$, we must have that
$\pr{z}{v} \leq \floor{h_K(z)}$ for all $v \in B$.
Now assume that for some $z \in \Z^n$, $\CC(K,S \cup \set{z}) \subsetneq \CC(K,S) = P$. We claim that $\pr{z}{v} > \floor{h_K(z)}$ for some $v \in
I$.
If not, then $\pr{v}{z} \leq \floor{h_K(z)}$ for all $v \in \ext(P)$, and hence $\CC(K,S \cup \set{z}) =
\CC(K,S)$, a contradiction. Hence such a $v \in I$ must exist.

For $z \in \Z^n$, note that $h_K(z) \geq h_{K \cap W}(z) = h_{K \cap W}(\pi_W(z))$.
Hence $\pr{z}{v} > \floor{h_K(z)}$ for $v \in I$ only if $\pr{\pi_W(z)}{v} = \pr{z}{v} > \floor{h_{K \cap
W}(\pi_W(z))}$. Let
$C' := \set{ w \in L^*,:\, \exists v \in I, \pr{v}{w} > \floor{h_{K \cap W}}(w)}$.
From the previous discussion, we see that $C \subseteq C'$.

Since $I \subseteq \relint(K) \cap W = \relint(K \cap W)$ we have  $\delta_v = \sup \set{r \geq 0: rB_n \cap W + v \subseteq K \cap W} > 0$ for all $v \in I$.
Let $\delta = \inf_{v \in I} \delta_v$. Since $|I| < \infty$, we see that $\delta > 0$. Now let $R = \frac{1}{\delta}$.
Take $w \in L^*$, $\|w\| \geq R$. Note that $\forall v \in I$,
\begin{equation}
\floor{h_{K \cap W}(w)} \geq h_{K \cap W}(w) - 1 \geq h_{(v + \delta B_n) \cap W}(w) - 1 = \pr{v}{w} + \delta\|w\|-1
\geq \pr{v}{w}.
\end{equation}
Hence $w \notin C'$. Therefore $C \subseteq C' \subseteq RB_n$ and $\CC(K)$ is finitely generated.
\end{proof}

\bibliographystyle{amsplain}
\bibliography{references,stuff}

\end{document}